\documentclass[11pt]{article}
\usepackage{amsfonts}
\usepackage[width=17cm, left=1cm]{geometry}
\usepackage{amssymb}
\usepackage{amsmath}
\usepackage{amsthm}

\newtheorem{definition}{Definition}
\newtheorem{proposition}{Proposition}
\newtheorem{theorem}{Theorem}

\newtheorem{corollary}{\noindent Corollary}

\begin{document}

\title{Duality and free measures in vector spaces, the spectral theory of actions of non-locally compact groups}

\author{A.~M.~Vershik
 \thanks{%
 St.~Petersburg Department of Steklov Institute of Mathematics, St.~Petersburg State University, Institute for Information Transmission Problems. \emph{E-mail}: vershik@pdmi.ras.ru.}
\thanks{%
Supported by the RFBR grant 17-01-00433.
}
}

\date{}
\maketitle

\rightline{\it \emph{To the memory of V.~N.~Sudakov}}

\begin{abstract}
The paper presents a general duality theory for vector measure spaces taking its origin in the author's papers written in the 1960s. The main result establishes a direct correspondence between the geometry of a measure in a vector space and the properties of the space of measurable linear functionals on this space regarded as closed subspaces of an abstract space of measurable functions. An example of useful new features of this theory is the notion of a free measure and its applications.

\end{abstract}

\tableofcontents

\section{Introduction}

The axiomatics of vector measure spaces discussed in this paper was conceived and partially described in the author's papers written in the 1960s, but its version presented below differs substantially from that given in those papers (\cite{V1,V2,V3,V4,V5}). The point is that a structure of a vector space, on which a measure and integration are considered, is introduced into a ready (and standard) measure space. That is, we introduce a structure of a vector space on an abstract measure space, and not the reverse.

This approach has numerous advantages compared with the traditional one, where a measure or integration are introduced at the last moment into a space overloaded with other structures, making some fundamental and difficult problems related to the measure 
even more complicated. The sixth volume of the Bourbaki series, {\it Integration}, is a telling illustration of how simple questions become complicated if one ignores categorical and probabilistic properties of measures. Our approach has a distinctly categorical nature, but, unfortunately, the category of measure spaces is usually ignored (presumably, because of its deceptive simplicity). A later realization of the same idea (see \cite{V} and subsequent papers) is to regard spaces endowed with a measure and a metric as spaces with a fixed measure and a varying metric (admissible triples). The most difficult and, in the author's opinion, still poorly studied part of measure theory is the geometry of families of $\sigma$-algebras. It was initiated by V.~A.~Rokhlin \cite{P}, who gave a classification of single $\sigma$-algebras (measurable partitions, measurable functions). The classification of vector measure spaces discussed in this paper is an example of a classification of approximately the same level of complexity. A much more difficult example of such a problem is the classification of filtrations (see \cite{V6}).

In the 1960s, there was a rather extensive literature on measure theory in infinite-dimensional vector spaces, a burst of interest in this field being evoked a little earlier by applications to physics (continuous integrals), functional analysis (generalized functions and generalized random processes), and the geometry of vector measure spaces proper (extension of measures, quasi-invariance of measures, etc.); see, e.g., \cite{V0,GGVil,Shil,Scor}  the references therein. The main achievement of those years in this direction was the important theorems due to V.~V.~Sazonov and R.~A.~Minlos on the extension of weak distributions to measures, initiated, respectively, by A.~N.~Kolmogorov and Yu.~V.~Prokhorov, and I.~M.~Gelfand: in a Hilbert space (Sazonov) and in a nuclear space (Minlos).

However, in those years, little attention was paid to similar classical probabilistic examples, as well as to connections to linear functional analysis and especially to general measure theory; as a result, the special nature of the situation with integration in infinite-dimensional spaces was somewhat overestimated: actually, general measure theory contained much of what is needed for constructing a special theory in such spaces. It is the latter fact that was emphasized in the author's papers on duality theory written in those years.

It  is relevant to recall that this paper on duality was presented at the International Congress of Mathematicians in Moscow in 1966
 \cite{V3}. The session was chaired by J.~Dieudonn\'e, the author of the remarkable fifth volume of the  Bourbaki series titled {\it Topological Vector Spaces} and concerned mainly with topological duality of locally convex vector spaces. He met the idea of extending the notion of duality to vector measure spaces with approval. On the contrary, L.~Schwartz, who worked at the time on the sixth volume titled {\it Integration}, did not accept these ideas.\footnote{After the congress, at which l made an acquaintance with J.~Dieudonn\'e, I conceived an idea which we implemented together with V.~N.~Sudakov: to write a letter to N.~Bourbaki paying compliments to this author for his enormous undertaking  of writing on the whole mathematics. We meticulously studied some of the Borbaki volumes at G.~P.~Akilov's seminar. I asked my friends, who had a very good knowledge of French, to translate the letter, and we sent it. After a while, each of us received as a gift a volume by Bourbaki. My copy was signed with two different pens, which I regarded as a confirmation of collective authorship. Soon, the sixth volume arrived, which, unfortunately, I did not like.}

It happened that in the late 1950s, V.~N.~Sudakov and I began to study measure theory in linear spaces. A wave of Moscow interest to this subject reached us through D.~A.~Raykov. V.\,N.\ started working, quite successfully, on quasi-invariant measures and other geometric problems (which afterwards evolved into working on the difficult problem of continuity of realizations of Gaussian processes, which he solved). In the late 1960s, V.\,N. and I wrote a large paper \cite{V0} on the general theory of extension of measures in vector topological spaces, which summarized our considerations on this subject. Even earlier, in the early 1960s, I started to think, proceeding from Rokhlin's theory of Lebesgue spaces, on structural measure theory in vector spaces, which resulted in the duality theory presented here in a revised form and the concept of a free vector measure. Brief notes and papers published in those years have not achieved sufficient prominence, though they were appraised by our seminar. The author hopes that this publication will, for one thing, revive the interest in important problems of  measure and integration theory in vector spaces and in the conceptual setting of these problems. Besides, the author intends to draw attention to new problems closely related to this theory, which are presented in the last section of the paper and which still remain poorly studied. I mean relations to representation theory and actions of non-locally compact groups. 

The author is grateful to V.~N.~Bogachev for his interest to the paper and to A.~A.~Lodkin for his help with the list of references.

I would like to repeat what I have said about V.\,N.\ in the sad days soon after his death.

 \vspace{5mm}

Vladimir Nikolaevich Sudakov (Volik, Volodya) is one of the most talented mathematicians of his generation. He blossomed early and brightly. During the whole siege, he remained in Leningrad with his parents. He graduated from school with a gold medal, being a winner of many olimpiads, and entered the Department of Mathematics and Mechanics of the Leningrad State University. It is then that we got acquainted. The easiness with which he grasped new mathematical ideas, the speed with which he solved problems inspired admiration. He was mainly attracted to geometric problems, but tried to learn all interesting and new things, to attend all kinds of courses. Eventually, he settled on functional analysis, as the most geometric branch of analysis in a wide sense, and our common first advisor was Gleb Pavlovich Akilov. Volodya formed a close friendship with him, and, undoubtedly, was influenced by him not only in mathematics. He wrote his first paper (published as a note in {\it Uspekhi Matematicheskikh Nauk}), with a simplification of the classical compactness criterion, being a fourth-year student. In the late 1950s, we were united by our interest in measure theory in functional spaces, a subject initiated by A.~N.~Kolmogorov and I.~M.~Gelfand that took root in Leningrad due to D.~A.~Raykov. We wrote a large survey paper on this subject, which has been repeatedly used afterwards. Volodya was an active member of our small seminar headed by Akilov (V.~P.~Khavin, B.~M.~Makarov, V.\,N., and I), which existed for several years and at which we discussed some Bourbaki volumes and other new things. At the beginning, V.\,N.\ took an active part in the ergodic seminar headed by V.~A.~Rokhlin (by the way, Rokhlin served as an opponent to the defense of both his theses). Later, V.\,N.\ offered him a co-authorship of a planned book with geometric content.

His post graduate advisor was L.~V.~Kantorovich, who at the time supervised the computational field in which Volodya majored at the University. Then he worked at the Leningrad (St.~Petersburg) Department of the Steklov Institute of Mathematics in the Laboratory of Statistical Methods headed by Yu.~V.~Linnik (and later by I.~A.~Ibragimov), but  while working on the theory of random processes, he preserved his own taste and circle of interests. He devoted a lot of time to the extremely difficult problem of continuity of realizations of a random Gaussian process and solved it by a method developed by himself. Another his widely known contribution is a theorem on the existence of an exact solution to the Monge--Kantorovich problem in the Euclidean case. A wide range of probabilists know his papers on Gaussian measures (written in part jointly with B.~Tsirelson and I.~A.~Ibragimov). And, of course, his monograph  \cite{S} achieved much prominence.  Note also that recently, his  paper  \cite{DS}, joint with L.~N.~Dovbysh, on Gram--de Finetti matrices has become quite popular.

 \section{Definitions}

\subsection{Definition of a vector measure space}

In what follows, we will consider vector spaces over the field of real numbers.

\begin{definition} Let $E$ be a vector space and assume that there is a  fixed linear space of homomorphisms of $E$ to the finite-dimensional vector spaces ${\Bbb R}^n$, $n \in \Bbb N$, which separates the points of $E$ (i.e., every point has a nonzero image under at least one homomorphism). Consider the collection of the inverse images of the Borel sets of ${\Bbb R}^n$ with respect to these homomorphisms and the $\sigma$-algebra $\mathcal A$ of subsets of $E$ generated by them. Assume that on 
 $\mathcal A$ there is a probability measure $\mu$ such that the triple $(E, {\mathcal A}, \mu)$ is isomorphic as a measure space to a Lebesgue--Rokhlin space. In this case, we say that $(E, {\mathcal A}, \mu)$ is a vector measure space.
\end{definition}

In fact, we have just defined the category of vector measure spaces: its objects are triples described above, and morphisms are measure-preserving linear homomorphisms. More exactly, by a morphism of objects we mean a linear map of spaces defined on a linear set of full measure and preserving the measure (or, in a somewhat different category, the type of the measure). We emphasize that this is a definition of morphisms \!\!\!$\mod 0$, i.e., the vector space itself is defined only up to choosing a vector subspace of full measure in the original space. Isomorphisms are defined in the same way as invertible morphisms.

In the category of vector measure spaces, one needs not care about special (e.g., topological) properties of spaces supporting the measure. But if a given vector space $E$ is also a locally convex topological separable space, then, as a system of $\sigma$-algebras in the above definition, one can take the $\sigma$-algebra of cylinder sets, i.e., sets generated by {\it continuous} linear functionals. This $\sigma$-algebra coincides with the $\sigma$-algebra of Borel sets, hence a locally convex separable space endowed with a Borel $\sigma$-additive measure is a vector measure space in our sense. It is slightly more difficult to check that an arbitrary separable metric linear topological space --- a so-called Fr\'echet space, not necessarily locally convex, endowed with a Borel probability measure --- is also a linear measure space in the sense described above.

The topology of vector spaces is not involved in our axiomatics and is used only in examples. We also do not consider  the important problem, once quite popular, of when a finitely additive measure (or, in the old terminology, a \textit{weak distribution}) defined on the algebra of cylinder sets of a vector (or vector topological) space can be extended in the same vector space to a countably additive measure, and deal only with countably additive measures.

\subsection{The space of linear measurable functionals}

Of importance for us is the following notion of a {\it linear measurable functional on a vector measure space}.

   \begin{definition}
A measurable function $f$ (more exactly, the class of functions coinciding \!\!\!$\mod 0$ with $f$) on a vector measure space $(E,\mu)$ is called a \textit{linear measurable functional} (l.m.f.)\ if there exists a vector subspace $E_f \subset E$ of full measure ($\mu(E_f)=1$)
on which some individual function from the class of $f$ is linear in the ordinary sense. Denote the space of l.m.f.\ by $\mbox{\rm{Lin}}(E,\mu)$. The homomorphisms from the definition of the $\sigma$-algebra of a vector measure space are linear measurable functionals, but, as a rule, they are far from exhausting the space
   $\mbox{\rm{Lin}}(E,\mu)$.
\end{definition}

Of course, as usually, by a functional we mean a class of functionals that pairwise coincide almost everywhere; the space $\mbox{\rm{Lin}}(E,\mu)$ is a subspace of the space $L_0(E,\mu)$ of all measurable functionals on $(E,\mu)$, and we endow it with the topology of convergence in measure. In some cases, for example, for Gaussian measures or for sequences of independent functionals, convergence in this topology coincides with convergence in the Hilbert space $L^2(E,\mu)$.

One sees from the definition that for a finite or countable set of linear measurable functionals, one can find a set of full measure on which all these functionals are linear. But the peculiarity of the notion of a l.m.f.\ is that in the most interesting case of infinite-dimensional spaces, the domain of definition of a linear measurable functional is individual: there exists no vector subspace of full measure on which all linear measurable functionals are defined simultaneously. In this, the duality theory described here differs fundamentally from the well-known and more complicated duality theory of topological vector spaces (see Vol.~V of N.~Bourbaki's treatise {\it Elements of Mathematics}). For more on these differences, see below. 

Consider a classical example. Let $\xi_n$, $n=1,2, \dots$, be a countable sequence of independent random variables taking the values $\pm 1$ with probabilities $1/2$. The Bernoulli measure $\mu$ corresponding to this sequence can be regarded as a measure on the linear space $l_{\Bbb N}^{\infty}$ (more exactly, on the set of vertices of the unit cube (ball) of this space). Endowing  $l_{\Bbb N}^{\infty}$ with the $\sigma$-algebra generated by the Borel sets in 
the weak $(l_{\Bbb N}^{\infty}, l_{\Bbb N}^1)$ topology, we obtain a vector measure space in the sense of the above definition.

By a corollary to the classical three-series theorem, the space of linear measurable functionals coincides with $l_{\Bbb N}^2$;
these functionals can be expressed as series         $$f(x)=\sum_{n\in \Bbb N} f_nx_n;\quad x=\{x_n\}\in
         l_{\Bbb N}^{\infty},\quad f=\{f_n\} \in l_{\Bbb N}^2,$$
that converge in measure $\mu$ (as well as almost everywhere and in the mean).\footnote{Observe that, not quite obviously, $l_{\Bbb N}^2$ is the space of  \textbf{all} measurable linear functionals. This does not follow from the three-series theorem, which deals only with series, but does follow from the structure of the whole space of measurable functions on an infinite product of spaces.} At the same time, the space of continuous linear functionals in the sense of the weak $(l_{\Bbb N}^{\infty}, l_{\Bbb N}^1)$  topology on $l_{\Bbb N}^{\infty}$ 
 is a smaller space, namely, $l_{\Bbb N}^1$.  A single space on which all measurable linear functionals are defined is $l_{\Bbb N}^2$ (regarded as a subset of $l_{\Bbb N}^{\infty}$), but it has zero measure. Thus there is no set of full measure on which all linear measurable functionals are defined. This example is typical.

The definition of a vector measure space immediately implies the following corollary.

    \begin{corollary}
In every vector measure space there are sufficiently many linear measurable functionals. 
   \end{corollary}

Namely, every such functional corresponds to a measurable map $E \rightarrow {\Bbb R}$ whose kernel is a subspace of codimension~$1$ (a hyperplane). Obviously, the fact that there are sufficiently many maps corresponding to subspaces of finite codimension implies that there are sufficiently many maps corresponding to hyperplanes. Moreover, in the definition one can also require the existence of sufficiently many hyperplanes only. 

The space $\mbox{\rm{Lin}}(E,\mu)$ is the main tool in duality theory. It can be called the ``measurable dual'' space to a linear measure space. To find an explicit description of this space for a given vector measure space $(E,\mu)$ is just as difficult as to describe the space of continuous linear functionals on a topological vector space.

\subsection{One-dimensional homomorphisms of vector measure spaces}

Consider a more general notion closely related to the notion of a linear measurable functional.

 \begin{definition}
 A measurable functional on a vector measure space is called a one-dimensional measurable linear map $f:E \rightarrow  {\Bbb R}$ if
  $$ f(\lambda_1 x + \lambda_2 y)=\lambda_1 f(x)+\lambda_2 f(x)$$
for all pairs  $(\lambda_1,\lambda_2)$ of real numbers (equivalently, almost everywhere with respect to the Lebesgue measure on ${\Bbb R}^2$) and almost all pairs $(x,y)$ with respect to the measure $\mu \times \mu$.
 \end{definition}
 
 In some papers, such functionals were called almost linear. Obviously, a measurable linear functional is a one-dimensional linear map of the space $(E,\mu)$ to ${\Bbb R}$. The converse is also true.

  \begin{theorem}
To every one-dimensional measurable linear map \!\!\!$\mod 0$ of a vector measure space to ${\Bbb R}$ there corresponds a unique  \!\!\!$\mod 0$ linear measurable functional on this space generating this map.
    \end{theorem}

Checking this fact is quite obvious. The difference in the understanding of linearity --- in the literal sense or for almost all pairs of values of the arguments --- can easily be overcome using Fubini's theorem, which allows one to make corrections ensuring literal linearity.


\section{Duality theory and classification}

In this section, we consider the basic facts of the duality theory of vector spaces endowed with probability measures.

\subsection{Characteristic functional}
The following classical definition underlies the whole measure theory. 

\begin{definition}
Consider the following functional, called the characteristic functional,
on the space $L_0(X,\mu)$ of all real measurable functionals of an arbitrary Lebesgue space (or on a subspace of $L_0(X,\mu)$):
$$ \xi(f)=\int_{x\in E} \exp\left(if(x)\right) d\mu(x), \quad f\in \mbox{\rm{Lin}}(E,\mu).
$$
\end{definition}

  {\it Obviously, $\xi(\cdot)$ is a positive definite functional on $L_0$ in the sense of the structure of the linear space, in particular, the Hilbert space
 $L^2$}. It completely determines the measure: the values of the measure on all measurable subsets generating the $\sigma$-algebra are uniquely defined. The restriction of the functional to $L^2$ is continuous in the topology of $L^2$.

The characteristic functional is the most convenient way (though rarely used beyond the probabilistic literature), among many others, to define an additional structure in the Banach spaces $L_0(X,\mu)$ or $L^p (X,\mu)$, $p\geq 0$, that characterizes them as spaces of measurable functions. For example, a short definition of the class of unitary operators in  $L^2(X,\mu)$ corresponding to transformations with invariant measure (i.e., multiplicative, real, preserving the partial order, etc.)\ is that it coincides with the class of unitary operators preserving the characteristic functional.


\subsection{Realization and classification of vector measure spaces}

Using the space $\mbox{\rm{Lin}}(E,\mu)$, we will give a universal realization of a vector measure space.

\begin{theorem}
Every vector measure space $(E,\mu)$ is linearly isomorphic to the space $({\Bbb R}^{\infty}, \nu)$ where $\nu$ is a Borel probability measure.
\end{theorem}

   \begin{proof}
Recall that $(E,\mu)$, regarded as a measure space, is, by definition, a Lebesgue space. Consider the space of all linear measurable functionals $\mbox{\rm{Lin}}(E,\mu)$ and choose in it a dense countable generating (i.e., separating the points) sequence of elements. Take a vector space $E_0$ of full measure on which these elements (as linear functionals) are defined and separate the points.

This sequence determines (as a sequence of random variables) an isomorphism between the spaces $(E,\mu)$ and ${\Bbb R}^{\infty}$.
Denote the image of $\mu$ under this isomorphism by $\nu$. The space
  $({\Bbb R}^{\infty},\nu) $ is, obviously, a vector measure space. But the constructed isomorphism is linear on the chosen set, since the collection of linear functionals is total and linear functionals are mapped to linear functionals.
 \end{proof}

\textbf{Remark.} The chosen method of constructing an isomorphism also proves that all linear measurable functionals for measures on 
${\Bbb R}^{\infty}$ correspond to series converging in measure. The difference between various ways of choosing the countable set from the proof of the theorem, obviously, leads to the above-mentioned difference between isomorphisms up to coincidence \!\!\!$\mod 0$. Of course, instead of
${\Bbb R}^{\infty}$ one can take other topological spaces, e.g., a Hilbert space (see below).

Let us turn to the question of what subspaces of the space of all measurable functions can be spaces of measurable linear functionals.

\begin{theorem}
Every generating subset of the space $L_0(X,\mu)$ of all measurable functions on a Lebesgue space $(X,\mu)$ closed in the topology of convergence in measure is the space of all measurable linear functionals on some vector measure space unique up to isomorphism.
\end{theorem}

\begin{proof}
Note that the case where $(X,\mu)$ is an atomic space, for example, with finitely many points, is trivial, though all arguments below are valid also for atomic measures. Hence all that follows concerns mainly the case where $(X,\mu)$ is a Lebesgue space with a continuous measure.

The following trick allows one to reduce the problem to Hilbert spaces. One can easily see that replacing the measure with an equivalent one, i.e., introducing a density on a vector measure space $(E,\mu)$, does not change the collection of linear measurable functionals. Hence, introducing a density allows one to consider only measures for which a dense countable set of linear measurable functionals is square-integrable, so that the intersection $\mbox{\rm{Lin}}(X,\mu)\bigcap L^2(X,\nu)$ becomes a generating subset. We have reduced the problem to the following one:

Prove that a given closed generating subspace $\mbox{\rm{Lin}}_0^2 \subset L^2(X,\nu)$ is the space of all square-integrable linear measurable functionals on some vector measure space.

This fact is an important refinement of Sazonov's theorem on measures in nuclear spaces (see the original papers by V.~V.~Sazonov and R.~A.~Minlos or, e.g.,  \cite{V0}).
\!\!\footnote{\,Minlos' theorem on extension of measures in nuclear spaces is less convenient for our purposes, but in principle it could also be used for the proof.}

Let us apply Sazonov's theorem to the real Hilbert space $H := \mbox{\rm{Lin}}_0^2 \subset L^2(X,\nu)$ on which a continuous positive definite functional $\chi$ is defined, namely, the characteristic functional (see the previous section). The theorem says that in a nuclear extension $\hat H$ of the space $\bar H \subset \hat H$ dual to $H$ there is a uniquely defined Borel probability measure $\mu$ whose Fourier transform coincides with $\chi$. Our refinement of this assertion is that $H$ is the space of all linear measurable functionals on $(\hat H, \mu)$, i.e.,
 $$\mbox{\rm{Lin}}(\hat H, \mu) = H \subset L^2(X,\nu)=L^2(\hat H, \mu).$$
 The inclusion of the right-hand side into the left-hand one is a part (usually omitted) of the assertion of Sazonov's theorem. Indeed, the fact that every element $h \in H$ is defined as a measurable linear functional on the nuclear extension follows from the fact that is lies (by definition) in the domain of definition of the Fourier transform of the measure $\mu$ --- this is exactly the assertion of the theorem --- i.e., determines a measurable functional on $(\hat H, \mu)$, and the linearity of this functional follows from its linearity (and continuity) on the space $\bar H$, which is dense in $\hat H$.

Let us show that there is no other linear measurable square-integrable functionals. Let
 $k\in L^2(\hat H, \mu)$ be such a functional. We may assume that it is orthogonal to $H$. Consider the direct sum
 $K\equiv H\bigoplus \{ck, c \in\Bbb R\}$, on which the same characteristic functional is defined, and apply to it the same theorem. We obtain a new vector measure space $(\hat K, \rho)$ (a nuclear extension of the space $\bar K$ dual to $K$). Clearly, it is a one-dimensional extension of the space $\bar H$. The projection $\bar K\rightarrow \bar H$ is a linear homomorphism of vector measure spaces and simultaneously a metric isomorphism, since the element $k$ is a function of elements of $H$, because $H$ is a generating space. On the other hand, this homomorphism is not a linear isomorphism, since, by our assumption, $k$ is a nonzero functional orthogonal to $H$. It follows that the linear functional $k$ on the space $\bar K$  is not a linear functional on the projection $\bar H$ (but, of course, it is some nonlinear functional)\footnote{A good and adequate geometric illustration of this argument is the above example with a Bernoulli measure on the product of two-point spaces: in this case,  $K=l^2$, $k=\{\frac{1}{2^n}\}$,
and $H$ is the orthogonal complement to $k$ in $K$. The projection sends the Bernoulli measure $\mu$ on the vertices of the unit cube isomorphically (i.e., uniquely)
 to the hyperplane $H$. However, with respect to the projection of $\mu$, the functional $k$ is  no longer linear. It is easy to express it in terms of Rademacher functions on the interval $[0,1]$, but this representation is not unique.}.

Thus we have proved that every closed subset in $L^2(X,\mu)$, where $(X,\mu)$ is a Lebesgue space with a continuous measure, is the space of all linear measurable functionals on a vector measure space.
 \end{proof}

  \subsection{Classification of vector measure spaces and spaces of measurable linear functionals}
 
 It follows from the previous theorem that defining a measure vector space satisfying the square-integrability condition for linear measurable functionals (in short,  a \emph{square-integrable measure vector space}) 
 is equivalent to defining a closed subspace in $L^2$: the linear isomorphism of such vector measure spaces is equivalent to the metric isomorphism of the spaces of measurable functionals on them.

 \begin{proposition}
 The classification of square-integrable measure vector spaces up to linear isomorphism \!\!\!$\mod 0$ is equivalent to the classification of the corresponding spaces of measurable linear functionals as subspaces in $L^2$ with respect to the group of unitary multiplicative operators (i.e., the group of measure-preserving transformations); or, equivalently (if we endow these spaces with the characteristic functionals defined above), to the classification of subspaces with respect to the group of unitary operators preserving the characteristic functional:
 $$(E_1,\mu_1)\sim (E_2,\mu_2) \quad\Leftrightarrow\quad (F_1,\xi)\simeq  (F_2,\xi),$$
where $F_1,F_2$ are the spaces of linear measurable functionals on $E_1,E_2$, respectively, and $\simeq$ stands for an isomorphism of these spaces as linear subspaces in $L^2$ of the form described above with respect to the group of unitary multiplicative operators. The corresponding isomorphism of vector measure spaces is exactly the isomorphism between  $(E_1,\mu_1)$ and $(E_2,\mu_2)$.
 \end{proposition}

Thus the type (up to isomorphism) of the space of all linear measurable functionals is a complete invariant (more exactly, coinvariant) with respect to the group of isomorphisms of vector measure spaces.

So, the two classifications are equivalent:

$\bullet$ classification of vector measure spaces;

$\bullet$ classification of generating closed linear subspaces in the space of all measurable functions on a Lebesgue space.

Of course, it is a kind of analog of the well-known metric classification of measurable functions (random variables): a complete metric invariant of a one-to-one \!\!\!$\mod 0$ (and hence generating) measurable function $f$ is the $f$-image of the measure (i.e., the distribution of $f$). 
In this case, a complete invariant of a closed generating subspace is the type of the space of linear measurable functionals.

A useful consequence of our considerations: every nondegenerate ($\chi(f)=1\Leftrightarrow f=0$) positive definite continuous functional on an infinite-dimensional real Hilbert space is the restriction of the characteristic functional under an appropriate embedding of this Hilbert space into the space $L^2([0,1],m)$ (where $m$ is the Lebesgue measure).

\smallskip
\textbf{Conclusion.}
\emph{The key aspect of the duality theory we propose is the parallel study of measures in vector spaces on one hand and closed generating subspaces in the space of all measurable functions on the other hand. A linear isomorphism of vector measure spaces is equivalent to a metric isomorphism of subspaces of measurable linear functionals.} 
\smallskip

Should one regard both these classifications as being ``wild'' or ``tame''? The answer to this question is not obvious, and perhaps the question itself is ill-posed. It reduces to the following question: to what extent the space of complete invariants of the classification is reasonable?

The classification of subspaces (e.g., of the space $L^2$) with respect to the group of multiplicative unitary operators, or, which is the same, the classification of orthogonal projections, is tame in the case of finite-dimensional subspaces or in the case where subspaces are subalgebras (in the latter case, the classification reduces to Rokhlin's classification of $\sigma$-subalgebras or measurable partitions). The author tends to believe that the case of an arbitrary generating subspace is also tame. Note that one may consider only generating spaces, since if the $\sigma$-algebra generated by a subspace is proper, then one should consider the problem in this subalgebra.


Nevertheless, it is most likely that one cannot provide a simple complete system of invariants of a subspace. But in the equivalent classification of measures in vector spaces up to linear isomorphism, there is also no visible canonical forms, unless we deal with special kinds of measures (products of independent measures, 
Markov measures, etc.). Even the example of finite-dimensional spaces shows that attempts to find a normal form of a measure are not hopeless, but unproductive. Still, both classification problems bear little resemblance to known ``nonsmooth'' problems, such as the classification of conjugacy classes of the group of measure-preserving automorphisms of the interval or the classification of dyadic filtrations \cite{V6}.

But one can study measures and their invariants corresponding to some remarkable subspaces in  $L^2$ and, equivalently, investigate the properties of some subspaces of measurable linear functionals for Gaussian measures, products of independent measures, 
etc. In the last section, we study a remarkable measure corresponding to the case where the space of linear functionals is the whole space of measurable functions, namely, develop the theory of free measures. A free measure is unique and universal.

\subsection{The kernel of a vector measure space and lifting}

In view of what we have said above on the domain of definition of a l.m.f., it is natural to pose the following question: can one define a single vector subspace of full measure on which all linear measurable functionals are defined? Of course, in a finite-dimensional measure space, the answer is positive, because every linear measurable functional is continuous. But we will see that in the essentially infinite-dimensional case, the answer is negative, which makes the duality theory of vector measure spaces meaningful.

\begin{definition}
The kernel $\mbox{\rm{Ker}}(E,\mu)$ of a vector measure space $(E,\mu)$ is the space of linear continuous (in the topology of convergence in measure) functionals on the space $\mbox{\rm{Lin}}(E,\mu)$ of linear measurable functionals. The kernel can be trivial, i.e., consist only of the zero functional; but it can also be complete, 
i.e., separate the points of the space $\mbox{\rm{Lin}}(E,\mu)$.
 \end{definition}

For example, it is easy to check that for Gaussian measures and for products of independent measures with finite variance, the kernel is complete. Clearly, the kernel $\mbox{\rm{Ker}}(E,\mu)$ is a subspace of $E$, since it is contained in each of the subspaces on which some  linear measurable functional is defined, i.e., lies in the intersection of all vector subspaces of full measure.

If the measure $\mu$ is such that every finite-dimensional subspace has zero measure, then the measure of the kernel is also zero: $\mu(\mbox{\rm{Ker}}(E,\mu))=0$.

For Gaussian measures, the kernel coincides with the intersection of all subspaces of full measure and has a structure of a Hilbert space nuclearly embedded into the original measure space; in fact, this follows from the Minlos--Sazonov theorem and the fact that the space $\mbox{\rm{Lin}}(E,\mu)$ is a closed subset in $L^2(E,\mu)$.

In this connection, note that the necessity to choose a set of full measure on which relations can be corrected to become identical
 (\!\emph{lifting}) appears in many problems of measure theory. For example, this is the case for a series of assertions on the continuity of measurable positive definite functions on groups and, in particular, on the continuity of measurable characters.

Even deeper are theorems on pointwise measurable actions of groups of unitary multiplicative operators, i.e., groups of classes of automorphisms coinciding \!\!\!$\mod 0$. For example, the group of all unitary multiplicative operators has no everywhere defined individual group action, this is proved in  \cite{GTW}, a partial result\,\footnote{\,%
 Namely, that there is no individual linear action of this group on a space with a free measure.} being obtained in an earlier paper  \cite{V5}.
At the same time, every locally compact group of unitary miltiplicative operators has an individual lifting; the proof is given in 
 \cite{Mc}. \marginpar{???}
 In \cite{V1,V5}, the author developed the idea of a free measure (see below) and gave another proof using structure theorems on locally compact groups and the Minlos--Sazonov theorem.

\subsection{On Gelfand triples}
In connection with the above argument, it makes sense to address the notion, having historical interest, of a Gelfand triple, used in  the 1950s--1970s.

Consider a triple of spaces
     $$ E \subset H (\sim H') \subset E'.$$
Here $E'$ is the space of Schwartz distributions on a compact finite-dimensional manifold. On  $E'$, a Borel probability measure is defined, e.g., a Gaussian measure $\mu$, and $E$ is the space of test functions, i.e., the space of all continuous linear functionals on the space of distributions $E'$. The space  $H$ (and $H'$) was originally defined as the Hilbert space with the norm  determined by the correlation operator corresponding to the Gaussian measure $\mu$.\footnote{\,This triple was called a Gelfand triple (trinity), and he interpreted it in the following way: $H$ (identified with $H'$) is God the Father, $E$ is God the Son, and  $E'$ is God the Holy Spirit. There are different ways to explain this expressive interpretation, but it is meant that the definition of the whole scheme begins with the correlation operators of the Gaussian measure, then we have a reliable definition of the space of test functions, and the space of distributions endowed with a measure --- actually, the main, but poorly manageable, object --- can be studied with the help of the first two spaces.} Observe, by the way, that $E$ is a countably normed nuclear space, hence, by Minlos' theorem, every weak distribution in $E'$ can be extended to a countably additive measure.

It turned out that in this scheme, the space $H$ (and the space $H'$ isomorphic to it), which appeared, rather formally, as a Hilbert space {\it with a weak distribution that is not a countably additive measure}, has a direct interpretation: this is exactly the 
 {\it space of all measurable linear functionals on the vector measure space $(E,\mu)$}. Some of them are continuous linear functionals (elements of $E'$) and thus are defined everywhere. But there are also others, which are defined only almost everywhere, each on its own vector subspace of full measure in $(E,\mu)$, but not on the whole space $E$. In this example, all linear measurable functionals are square-integrable, $H\subset L^2(E,\mu)$, and the topology of convergence in measure coincides with the Hilbert topology (of $L^2$). The embedding $E'\subset H'$ of the space of continuous linear functionals into the space of measurable linear functionals is dense and (by Sazonov's theorem) is a nuclear embedding (i.e., is realized by a nuclear operator; sometimes, one speaks about a Hilbert--Schmidt embedding, meaning the square of the original nuclear operator).

It is useful not to identify the space $H'$, the Hilbert space dual to $H$, with $H$; its independent interpretation is described above: this is the so-called kernel of the Gaussian measure, i.e., the space $H'$ is the invariantly defined intersection of all vector subspaces of full measure $\mu$ in $E'$, which itself has zero measure. The space $H$ is exactly the space of all measurable linear functionals on $(E,\mu)$ in the case of Gaussian measures.

\section{A free measure space}

\subsection{Definition of a free measure}

\begin{definition}
A measure in a vector space is called free if every measurable functional is linear up to a correction on a set of zero measure depending, in general, on the functional.
\end{definition}

This important notion was implicitly used in  \cite{Mc}; 
 its exact description and the term 
 \emph{free measure} appeared in the author's papers \cite{V1,V5}.  The main purpose of this notion is rather methodological and categorical, but it turned out to be useful also in a number of practical questions.

From the viewpoint  of the classification theorem from the previous section, by which to every closed generating subspace there corresponds a measure, a free measure corresponds to the whole space of measurable functions.

However, it is useful to give a more direct construction than the one following from the previous, in a sense implicit, constructions.

Obviously, a vector space with an atomic measure supported exactly by the unit coordinate vectors of this space has this property: every measurable functional coincides almost everywhere with the linear functional determined by the values on the unit coordinate vectors. This is an example of an atomic vector space with a free measure. For atomic measures, all subsequent assertions are obvious, since a geometric realization of a free atomic measure is unique.

\subsection{Direct constructions of spaces with a continuous free measure}

Consider the vector space ${\Bbb R}^{\Bbb Z}$ of all two-sided infinite sequences of real numbers endowed with the natural $\sigma$-algebra $\mathcal B$ of Borel sets, and map the circle (actually, the interval) into this space:
$$ I:{\Bbb T} \rightarrow {\Bbb R}^{\Bbb Z}; \quad I(\lambda) =\{\exp{2\pi\cdot i n \lambda}\}_{n \in \Bbb Z}.$$

Let $\omega=I_*(m)$ be the image of the Lebesgue measure $m$ on $[0,1)$. Clearly, 
 $(R^{\Bbb Z},{\mathcal B},\omega)$ is a linear measure space in the sense of our definition. The existence of sufficiently many linear measurable (and even continuous in the weak topology of the space ${\Bbb R}^{\Bbb Z}$) functionals is obvious, but a more interesting fact holds, which follows from a difficult classical theorem on trigonometric series.

\begin{proposition}
Every measurable (almost everywhere finite) function on the measure space $({\Bbb R}^{\Bbb Z},{\mathcal B},\omega)$ lying in  $L^2$ is \!\!\!$\mod 0$ linear, i.e., coincides with a linear functional on a set of full measure.
\end{proposition}

\begin{proof} The well-known Menshov theorem (see \cite[Chap.~XV, \S\,2]{Ba}) says that every measurable function on the interval
 $(0,2\pi)$ is the sum of a (nonunique, in general) trigonometric series converging almost everywhere, which does not always coincide with the Fourier transform of its sum.
 
By construction, the measure space  $({\Bbb R}^{\Bbb Z},\omega)$ is isomorphic to the space $([0,1),m)$, hence every function $f$ from $L^2({\Bbb R}^{\Bbb Z},\omega)$ is the image of a function from $L^2([0,1),m)$ under the map conjugate to the map
 $I$:  $I^*(f)(\cdot)=f(I^{-1}(\cdot))$. But the function $f$ goes to the sequence of its Fourier coefficients:
$$\{\hat f_n\}_{n \in \Bbb Z}=\bigg\{\int_0^1 e^{2\pi i n \lambda}\cdot f(\lambda)d\lambda\bigg\}_{n \in \Bbb Z}.$$
On the other hand, the sequence $\{{\hat f}_n\}$ determines a linear measurable function on the linear measure space $({\Bbb R}^{\Bbb Z},\omega)$ by the formula
$$\langle\hat f,x\rangle:= \sum _{n \in \Bbb Z} {\hat f}_n x_n$$
for $\omega$-almost all (but not all) $x \in {\Bbb R}^{\Bbb Z}$.
\end{proof}

Thus we have constructed a space $({\Bbb R}^{\Bbb Z},\omega)$ with a continuous free measure.

For other constructions of a vector space with a continuous free measure, it suffices to find a countable system of measurable functions
$\{f_n\}$, $n \in N$, on $[0,1]$ (or on another Lebesgue space with a continuous measure) such that every measurable function on this space can be written as a series in this system converging almost everywhere on $[0,1]$.

Yet another construction is as follows. Consider a nuclear extension $\hat H$ of the space $L^2([0,1],m)$; it contains generalized eigenfunctions of the operator of multiplication by every bounded measurable function, i.e., it contains the $\delta$-functions
 $\delta_x$, $x\in [0,1]$. Transfer the measure $m$ over to the set of generalized eigenfunctions by the tautological map  $x\mapsto \delta_x$, and we obtain a free measure on the space~$\hat H$.

All these examples are isomorphic, as follows from the theorem on the isomorphism of vector measure spaces having the same space of measurable linear functionals.

Other models of a free measure in a vector space, adapted to different probabilistic or operator problems, are also of interest.

\subsection{Properties of vector spaces with a free measure}

We list, with brief explanations, a number of useful properties of free measures.

\begin{theorem}
{\rm1.} Every vector measure space is a linear isomorphic \!\!\!$\mod 0$ image of a vector space with a free continuous measure.
Thus a free measure plays the role of a universal object in the category of vector measure spaces.

{\rm 2.} Every measure-preserving automorphism of a space with a free measure is a \!\!\!$\mod 0$ linear automorphism of the space. Thus the group of all automorphism classes of a Lebesgue space with a continuous measure is a subgroup of the group of linear automorphisms of a space with a free measure.

{\rm3.} Every measurable partition of a space with a free measure coincides \!\!\!$\mod 0$ with the partition into the cosets of some vector subspace. Thus the lattice of all classes of measurable partitions of a Lebesgue spaces with a continuous measure is a sublattice of the lattice of affine partitions of the vector space.
\end{theorem}

\begin{proof}
1. This follows from the general fact that a space of linear measurable functionals $\mbox{\rm{Lin}}(E_1,\mu_1)$ is contained 
in another such space $(E_2,\mu_2)$ if and only if the space $(E_1,\mu_1)$ is isomorphic to a quotient space of  $(E_2,\mu_2)$. One can easily check this using the model of a vector measure space described above. The measurable partition with respect to which the quotient is taken is exactly the partition generated by the image of the smaller subspace.

2. Follows from the definition of a free measure. For details, see \cite{V1,V5}.

3. The subalgegra in  $L^{\infty}$ of functions constant on the elements of the partition determines an affine measurable partition of the space with a free measure. The fact that it is isomorphic to the original partition follows from the coincidence of the subalgebras of functions constant on the elements of the partitions.
\end{proof}

The kernel of a space with a free measure, in the sense of the definition from the previous section, is the zero subspace, since it is this subspace that is the common domain of definition of all linear measurable functionals.

The notion of a free measure allows one to establish a correspondence between geometric constructions in measure spaces and their functional analytic analogs in spaces of measurable functions. This is the main purpose of this notion.

\section{Applications: quasi-invariant measures and the spectral theorem}

\subsection{Quasi-invariant measures in vector spaces}

\begin{definition}
A measure in a vector space is called quasi-invariant if its kernel is complete (i.e., separates the measurable linear functionals) and the translation by an arbitrary element of the kernel sends the measure to an absolutely continuous measure.
\end{definition}

A Gaussian measure is quasi-invariant in this sense. This observation goes back to the well-known Cameron--Martin theorem of the 1940s on the Wiener measure, and was repeatedly rediscovered for different reasons (J.~Feldman's theorem on equivalent Gaussian measures). Around Gelfand's seminar in the late 1950s, this subject was especially popular and took the shape of the Gelfand--Segal scheme.  Representations of the canonical commutation relations and many other things relied upon this property of Gaussian measures. In this connection, we may mention also V.~N.~Sudakov, I.~V.~Girsanov, and others. Non-Gaussian quasi-invariant measures are also known (A.~V.~Skorokhod, A.~M.~Vershik, and others). We will state only one property, briefly considered in 
\cite{V2} and related to the space of linear measurable functionals. To every translation that sends the measure to an absolutely continuous measure there corresponds a linear continuous operator in the space of all measurable functions that sends a linearly measurable functional to its sum with some constant depending linearly on the functional. In other words, an admissible translation determines a continuous linear functional on the space of linear measurable functionals, i.e., a point of the kernel:
$$(T_hf)(x)=f(x+h)=f(x)+f(h), \quad U_{T_h}f=f+\langle c,h\rangle1. $$
Multiplying the operator $U_{T_h}$ by the square root of the density of the transformed measure with respect to the original one, we obtain a unitary operator in $L^2$. This provides a potential method for describing quasi-invariant measures as measures for which the space of linear measurable functionals admits an operator that can be extended to an operator in the whole space of measurable functions (or in $L^2$).

\subsection{Spectral measures of non-locally compact Abelian groups}

The spectral theory of Abelian locally compact groups of unitary operators is based on the theory of characters for these groups: the spectral measure of a unitary representation of such a group is a measure on the group of its continuous characters; here the general theory for an arbitrary Abelian locally compact group is little different from the case of the groups $\Bbb Z$ and $\Bbb R$.

When passing to non-locally compact Abelian groups, the situation changes, since the group of continuous characters in the ordinary sense, i.e., the group of continuous positive homomorphisms of the group to the circle, even when there are sufficiently many characters (i.e., they separate the elements of the group, which is not always the case), may not suffice for constructing the spectral theory of representations. The reason of this is exactly what was discussed above: there may be ``more'' positive definite functions on the group than probability measures on the group of continuous characters. In other words, Bochner's theorem does not hold. For example, for the additive group of an infinite-dimensional real Hilbert space $H$, Sazonov's theorem determines an additional condition on a positive definite function $f$ on $H$ that guarantees that it is the Fourier transform of some measure on the dual space (i.e., on the group of characters): this condition is exactly the continuity of $f$ in the nuclear topology on $H$, i.e., a strengthened continuity.
Note that for other Banach spaces, in general there are no such conditions, or they are not purely topological (i.e., conditions of continuity in some weakened topology). 

But the same theorem yields a recipe for constructing the spectral measure for cyclic representations of the additive group of an infinite-dimensional Hilbert space. The following theorem on the spectral decomposition of a unitary representation of the additive group of a Hilbert space is a corollary of Sazonov's theorem stated in an appropriate form.

\begin{theorem}
Every continuous unitary cyclic representation $\pi$ (i.e., a representation with a cyclic vector) of the additive group of an infinite-dimensional separable Hilbert space $H$ by operators in some separable Hilbert space $\cal H$,
$$ g \mapsto \pi(g), \quad g \in H,$$ 
admits a unique spectral realization as a direct integral with one-dimensional fibers, i.e., as a space
 $L^2_{\mu}(\bar H; \Bbb C)$ where $\bar H$ is an arbitrary nuclear extension of  $H$ and $\mu$ is a measure (spectral measure) in the space $\bar H$ defined up to absolute continuity; it is completely defined if in $\pi$ a cyclic element of norm~$1$ is chosen:
$$ {\cal H}=\int_{\bar H}^\oplus \Bbb C\, d\mu.$$
The representation operators are multiplicators:
    $$(\pi(g)F)(h)=e^{i\langle h,g\rangle}F(h),\quad F \in L^2(\bar H,\mu)= {\cal H},\quad h\in \bar H, $$
where the bilinear pairing $\langle h,g\rangle$  gives the value of a measurable linear functional $g$ on an element $h$ of the measure space $H$.
 \end{theorem}

Thus the difference with the classical theorem on the spectral decomposition of a unitary cyclic representation of an Abelian group is that the group of continuous additive characters over which the direct integral is taken is replaced with the group of measurable additive characters (= linear measurable functionals) studied in this paper.

The proof of the theorem and the properties of this decomposition are the same as in the classical locally compact case. The choice of the additive group of a Hilbert space as an example is due to its importance in the representation theory of non-locally compact groups (see, e.g., \cite{GGV,GGV2,Al,Go}).
 On the other hand, as we have seen, in this case the group of measurable characters has a simple description.

We conclude this section with the remark that the definition and properties of ordinary operations on measures (convolution, convex combinations, etc.), needed for the analysis of spectral measures of unitary representations, are no different from the same definitions and properties in the case of continuous characters, as long as measurable characters (= linear measurable functionals) are well defined. 

These considerations are important for the representation theory of non-locally compact groups, in particular, of non-Abelian groups of currents (see, e.g., \cite{GGV,Al}). 

\subsection{Actions of non-locally compact groups with a quasi-invariant measure and von Neumann factors}

From the viewpoint of the theory of dynamical systems, the study of actions of non-locally compact Abelian groups with a quasi-invariant measure like that considered above poses an entirely new problem, which has not apparently been investigated. In our example, this is the action of the additive group of the Hilbert space $l^2$ by translations on the elements of the kernel in the space ${\Bbb R}^{\infty}$ with the standard Gaussian measure. The problem is how to classify
orbit partitions of such actions. In the above example,  the partition in question is the partition of a space (a Hilbert space or ${\Bbb R}^{\infty}$) with a Gaussian measure into the cosets of the subspace $l^2$. This partition of a measure space is well defined  $\mod 0$; of course, it is not measurable, but it is also not isomorphic to any orbit partition of a locally compact group acting with a quasi-invariant measure.

It is also obvious that this partition is not hyperfinite (or, in other terms, ``tame''), i.e., is not the limit of a decreasing sequence of measurable partitions. The paradoxicality of this fact is that, as is well known, all actions with an invariant or quasi-invariant measure of locally compact Abelian and even amenable groups are tame (see \cite{CFW}). This even caused the author of  \cite{To} to suppose that the additive group of the space $l^2$ is nonamenable. This is hardly justified, since Abelian groups (even not locally compact) are, after all, amenable in the original (von Neumann's) sense: they carry an invariant mean. 

Recall that ergodic actions of amenable locally compact groups with an invariant measure are orbit isomorphic; this is the above-mentioned generalization of the well-known Dye theorem. Actions of such groups with a quasi-invariant measure are tame 
  \cite{CFW}, and there is a satisfactory classification of their orbit partitions and their invariants (such as the list of relations), see papers by A.~Connes, W.~Kriger, H.~Araki and Woods, A.~Vershik, and others). The question arises about an analog of these theorems, albeit in a very special setting: what are orbit invariants of ergodic actions of the additive group of $l^2$ with a quasi-invariant measure? Moreover, are ergodic actions of  the additive group of $l^2$ with different Gaussian measures isomorphic?

The problem is closely related to the classification of regular subalgebras in von Neumann factors. More exactly, consider the Koopman and von Neumann unitary representations generated by such an action. Operators of the Koopman representation act in the space
$L^2(\bar H)$ by shifts of the argument with simultaneous multiplication by the square root of the corresponding density. The spectral theory of this representation for the case of $l^2$ can be analyzed by the method described in the previous section (Sec.~5.2). The second, von Neumann, representation has a more complicated structure, it does not fit into the usual scheme of representations of skew products, since for non-locally compact groups there is no natural group  $C^*$-algebra. For locally compact groups, and, more generally, for locally compact groupoids, there is a construction generalizing the original von Neumann's construction; it defines the factor and its commutant generated by the right and left actions of the group elements and the operators of multiplication by bounded measurable functions. 

But our case differs substantially from the locally compact one in that the Hilbert space in which factors are defined is not the $L^2$ space constructed from a measure supported by the orbit equivalence relation, i.e., this is not the classical groupoid construction. Though a groupoid can be defined from the action in exactly the same way, but the measure of the equivalence relation  (regarded as a subset of the direct product of the space $\bar H$ with itself) is equal to zero, hence there is no expectation operator, and the construction of the factor is itself nontrivial. This construction is sketched in  \cite{V5}. It can also be realized via polymorphisms. The point is that the corresponding regular maximal self-adjoint Abelian subalgebra (MASA) is not a Cartan subalgebra, for which the conditional expectation exists. To the best of the author's knowledge, such subalgebras have not been studied.

The problem of establishing the hyperfiniteness of such factors and their type is already not quite trivial. In our case, it is not difficult to prove that the factor is hyperfinite. Most likely, it is of type~${\rm III}_1$, but apparently this has not yet been proved.

More generally, the question is about nonstandard maximal regular subalgebras in hyperfinite factors of type~III without conditional expectations. This is apparently a new circle of problems in the theory of von Neumann algebras. The author thanks R.~Longo and S.~Popa for confirming this opinion.

\end{document}